\documentclass{amsart}
\usepackage{graphicx}
\usepackage{pslatex}
\usepackage{amsmath,amsfonts}
\usepackage{rotating}
\usepackage{amssymb}
\usepackage{verbatim}
\usepackage{rotating}
\usepackage{color}
\usepackage{mathrsfs}
\usepackage{hyperref}

\newtheorem{theorem}{Theorem}[section]
\newtheorem{lemma}[theorem]{Lemma}
\newtheorem{proposition}[theorem]{Proposition}

\theoremstyle{definition}
\newtheorem{definition}[theorem]{Definition}
\newtheorem{example}[theorem]{Example}

\theoremstyle{remark}

\numberwithin{equation}{section}

\def\R{\mathbb R}
\def\C{\mathbb C}

\def\supp{\text{supp}}
\def\({\left(}
\def\){\right)}
\def\[{\left[}
\def\]{\right]}
\def\<{\left<}
\def\>{\right>}
\def\less{\lesssim}

\begin{document}

\title{Weighted Multilinear Square Functions Bounds}

\author{Lucas Chaffee}
\address{Department of Mathematics\\ University of Kansas\\ Lawrence  KS 66045}
\email{ lchaffee@math.ku.edu}
\author{Jarod Hart}
\address{Department of Mathematics\\ University of Kansas\\ Lawrence  KS 66045}
\email{ jhart@math.ku.edu}
\author{Lucas Oliveira}
\address{Departamento de Matem\'{a}tica\\UFRGS\\ Porto Alegre RS 91509-900}
\email{lucas.oliveira@ufrgs.br}

\thanks{Chaffee was supported in part by NSF Grant \#DMS1069015.}
\thanks{Hart was supported in part by NSF Grant \#DMS1069015.}
\thanks{Oliveira was supported in part by CAPES-Processo 2314118}


\subjclass[2011]{Primary 42B02; Secondary 44A02}

\date{February 14, 2012.}

\dedicatory{ }

\keywords{Square Function, Littlewood-Paley, Bilinear, Calder\'on-Zygmund Operators}

\begin{abstract}
In this work we study boundedness of Littlewood-Paley-Stein square functions associated to multilinear operators.  We prove weighted Lebesgue space bounds for square functions under relaxed regularity and cancellation conditions that are independent of weights, which is a new result even in the linear case.  For a class of multilinear convolution operators, we prove necessary and sufficient conditions for weighted Lebesgue space bounds.  Using extrapolation theory, we extend weighted bounds in the multilinear setting for Lebesgue spaces with index smaller than one.
\end{abstract}

\maketitle

\section{Introduction}

Given a function $\psi:\R^n\rightarrow\C$, define $\psi_t(x)=t^{-n}\psi(t^{-1}x)$ and the associated Littlewood-Paley-Stein type square function
\begin{align}
g_\psi(f)=\(\int_0^\infty|\psi_t*f|^2\frac{dt}{t}\)^\frac{1}{2}.\label{firstsqfunction}
\end{align}
These convolution type square functions were introduced by Stein in the 1960's, see e.g. \cite{S1} or \cite{S2}, and have been studied extensively since then, including classical works by Stein \cite{S1}, Kurtz \cite{K}, Duoandikoetxea-Rubio de Francia \cite{DRdF}, and more recently Duoandikoetxea-Seijo \cite{DS}, Cheng \cite{Che}, Sato \cite{Sa}, Duoandikoetxea \cite{D1}, Wilson \cite{W}, Lerner \cite{L}, and Cruz-Uribe-Martell-Perez \cite{CUMP}.  Of particular interest of these, \cite{K}, \cite{DS}, \cite{Sa}, \cite{W}, \cite{CUMP}, and \cite{L} prove bounds for $g_\psi$ on weighted Lebesgue spaces under various conditions on $\psi$.  Non-convolution variants of \eqref{firstsqfunction} were studied by Carleson \cite{C}, David-Journ\'e-Semmes \cite{DJS}, Christ-Journ\'e \cite{CJ}, Semmes \cite{Se}, Hofmann \cite{Ho1,Ho2}, and Auscher \cite{A} where they replaced the convolution $\psi_t*f(x)$ with 
\begin{align*}
\Theta_tf(x)=\int_{\R^n}\theta_t(x,y)f(y)dy.
\end{align*}
In \cite{DJS} and \cite{Se}, the authors proved $L^p$ bounds for square Littlewood-Paley-Stein square functions associated to $\Theta_t$ when $\Theta_t(b)=0$ for some para-accretive function $b$.  In \cite{Ho1,Ho2}, this type of mean zero assumption is replaced by a local cancellation testing condition on dyadic cubes.  In \cite{C}, \cite{CJ}, and \cite{A}, the authors replace mean zero assumption with a Carleson measure condition for $\theta_t$ to prove $L^2$ bounds for the square function.  The work of Carleson in \cite{C} was phrased as a characterization of $BMO$ in terms of Carleson measures, but non-convolution type square function bounds are implicit in his work.

In all of the works studying $g_\psi$ cited above, the authors assume that $\psi$ has mean zero.  In fact, if $g_\psi$ is bounded on $L^2$, then $\psi$ must have mean zero, but in the non-convolution setting, the mean zero condition is no longer a strictly necessary one, as demonstrated in \cite{C}, \cite{Ho1}, \cite{Ho2}, and \cite{A}.  This phenomena persists in the multilinear square function setting, and in this work we explore subtle cancellation conditions for multilinear convolution and non-convolution type square function and their interaction with weighted Lebesgue space estimates.

The non-convolution form of the kernel $\theta_t(x,y)$ allows for a natural extension to the multilinear setting.  Define for appropriate $\theta_t:\R^{(m+1)n}\rightarrow\C$
\begin{align}
&S(f_1,...,f_m)(x)=\(\int_0^\infty|\Theta_t(f_1,...,f_m)(x)|^2\frac{dt}{t}\)^\frac{1}{2},\text{ where }\label{sqfunction}\\
&\Theta_t(f_1,...,f_m)(x)=\int_{\R^{mn}}\theta_t(x,y_1,...,y_m)\prod_{i=1}^mf_i(y_i)d\vec y\label{theta}
\end{align}
where we use the notation $d\vec y=dy_1\cdots dy_m$.  When $m=1$, i.e. in the linear setting, this is the operator $\Theta_t$ mentioned above, so we use the same notation for it.  We wish to find cancellation conditions on $\theta_t$ that imply boundedness $S$, given that $\theta_t$ also satisfies some size and regularity estimates.  In particular, we assume that $\theta_t$ satisfies
\begin{align}
&|\theta_t(x,y_1,...,y_m)|\less\prod_{i=1}^m\frac{t^{-n}}{(1+t^{-1}|x-y_i|)^N}\label{size}\\
&|\theta_t(x,y_1,...,y_m)-\theta_t(x,y_1,...,y_i',...,y_m)|\less t^{-mn}(t^{-1}|y_i-y_i'|)^\gamma\label{regy}
\end{align}
 for all $x,y_1,...,y_m,y_1',...,y_m'\in\R^n$ and $i=1,...,m$ and some $N>n$ and $0<\gamma\leq1$.  Note that we do not require any regularity for $\theta_t(x,y_1,...,y_m)$ in the $x$ variable.  Square functions associated to this type of operators have been studied in a number of recent works.  In Maldonado \cite{M} and Maldonado-Naibo \cite{MN}, the authors introduce the operators \eqref{theta}, and making the natural extension of Semmes's point of view in \cite{Se} to prove bounds for a Besov type relative of the square function $S$ \eqref{sqfunction},
\begin{align*}
(f_1,...,f_m)\mapsto\(\int_0^\infty||\Theta_t(f_1,...,f_m)||_{L^p}^2\frac{dt}{t}\)^\frac{1}{2}.
\end{align*}
When $p=2$ this Besov type square function agrees with the square function \eqref{sqfunction}.  In \cite{Hart1}, \cite{GO}, and \cite{GLMY}, Hart, Grafakos-Oliveira, and Grafakos-Lui-Maldonado-Yang proved boundedness results for different versions of the square function $S$ in Lebesgue spaces under various cancellation and regularity conditions on $\theta_t$.  That is, in each of these works the authors proved bounds of the form $||S(f_1,...,f_m)||_{L^p}\less||f_1||_{L^{p_1}}\cdots||f_m||_{L^{p_m}}$, for minor modifications of $S$ in various ranges of indices $p,p_1,...,p_m$.  The first goal of this work includes proving a weighted version of these results,
\begin{align}
||S(f_1,...,f_m)||_{L^p(w^p)}\less\prod_{i=1}^m||f||_{L^{p_i}(w_i^{p_i})}\label{Lpbound}
\end{align}
for appropriate $1<p_1,...,p_m<\infty$, $w_i^{p_i}\in A_{p_i}$ and $w=w_1\cdots w_m$.  More generally, the main result of this work is the following theorem.

\begin{theorem}\label{t:main}
Assume $\theta_t$ satisfies \eqref{size} and \eqref{regy}.  Then the following cancellation conditions are equivalent
\begin{itemize}
\item[\it i.] $\Theta_t$ satisfies the strong Carleson condition,
\item[\it ii.]  $\Theta_t$ satisfies the Carleson and two cube testing conditions.
\end{itemize}
Furthermore, if the equivalent conditions {\it (i)} and {\it (ii)} hold, then $S$ satisfies \eqref{Lpbound} for all $w_i^{p_i}\in A_{p_i}$ where $w=w_1\cdots w_m$, $1<p_1,...,p_m<\infty$ satisfying $\frac{1}{p}=\frac{1}{p_1}+\cdots+\frac{1}{p_m}$, and $f_i\in L^{p_i}(w_i^{p_i})$.
\end{theorem}

For the definitions of the Carleson, strong Carleson, and two-cube testing conditions, see Section 3.  For now we only note that conditions quantify some cancellation of $\theta_t$ and that $\Theta_t(1,...,1)=0$ for all $t>0$ implies all three of these conditions.  It is of interest to note that there is no mention of weighted estimates in the hypotheses of Theorem \ref{t:main}, but we conclude boundedness of $S$ in weighted Lebesgue spaces.  Also this is the first result for multilinear square functions of this type where $S$ is bounded for $1/m<p<2$ and $\Theta_t(1,...,1)$ is not necessarily zero for all $t$.

An approach that has been used to prove bounds for $S$ with $1/m<p\leq1$ is to view $\{\Theta_t\}_{t>0}$ as a Calder\'on-Zygmund taking values in $L^2(\R_+,\frac{dt}{t})$, and reproduce the classical Calder\'on-Zygmund theory to prove a weak endpoint bound and interpolate with bounds for $p>1$.  But in order for $\{\Theta_t\}_{t>0}$ to be a Calder\'on-Zygmund operator, one must require a regularity condition in the first variable of $\theta_t$.  In this paper, we use almost orthogonality estimates and Carleson type bounds adapted to a weighted setting, and extend bounds to indeces $p<1$ by the weight extrapolation of Grafakos-Martell \cite{GM}.

We also prove a stronger result for square functions associated to a certain class of multiconvolution operators.  We prove necessary and sufficient cancellation conditions for bounds of $S$ when $\Theta_t$ is given by convolution for each $t>0$.  As a consequence, we also provide a classical Calder\'on-Zygmund type analogue for square functions:  If $\Theta_t$ is given by convolution for each $t$ and $S$ is bounded on $L^{p_0}$ for some $p_0\geq2$, then $S$ is bounded on all reasonable weighted Lebesgue spaces, including spaces with index smaller than one in the multilinear setting.  We state these results precisely in the following theorem.

\begin{theorem}\label{t:mconv}
Suppose $\theta_t(x,y_1,...,y_m)=t^{-mn}\Psi^t(t^{-1}(x-y_1),...,t^{-1}(x-y_1))$ satisfies \eqref{size} and \eqref{regy} for some collection of functions $\Psi^t:\R^{mn}\rightarrow\C$ depending on $t>0$.  Then the following are equivalent
\begin{itemize}
\item[\it i.]  $\Theta_t$ satisfies the Carleson condition
\item[\it ii.]  $S$ satisfies the unweighted version of \eqref{Lpbound} for some $1<p_1,...,p_m<\infty$ and $2\leq p<\infty$ that satisfy $\frac{1}{p}=\frac{1}{p_1}+\cdots+\frac{1}{p_m}$, that is \eqref{Lpbound} with $w_1=\cdots=w_m=w=1$
\item[\it iii.]  $S$ satisfies \eqref{Lpbound} for all $1<p_1,...,p_m<\infty$ that satisfy $\frac{1}{p}=\frac{1}{p_1}+\cdots+\frac{1}{p_m}$, $w_i^{p_i}\in A_{p_i}$ where $w=w_1\cdots w_m$, and $f_i\in L^{p_i}(w_i^{p_i})$.
\item[\it iv.]  $\Theta_t$ satisfies the strong Carleson condition
\end{itemize}
Furthermore, if $\Psi^t=\Psi$ is constant in $t$, then conditions {\it (i)}-{\it (iv)} are equivalent to $\Theta_t(1,...,1)=0$ as well.
\end{theorem}

We organize the article in the following way:  In Section 2, we prove the some convergence results and boundedness results for $S$ when $\Theta_t(1,...,1)=0$.  In Section 3, prove various properties relating the Carleson, strong Carleson, and two cube testing conditions to each other and some bounds for $S$.  Finally in section 4, we prove Theorems \ref{t:main} and \ref{t:mconv}.

\section{A Reduced T1 Theorem for Square Functions on Weighted Spaces}

It is well-known that \eqref{size} implies that $|\Theta_t(f_1,...,f_m)(x)|\less Mf_1(x)\cdots Mf_m(x)$, where $M$ is the Hardy-Littlewood maximal function, and hence
\begin{align*}
\sup_{t>0}||\Theta_t(f_1,...,f_m)||_{L^p}\less\prod_{i=1}^m||f_i||_{L^{p_i}}
\end{align*}
when $1<p_1,...,p_m<\infty$ satisfy the H\"older type relationship
\begin{align}
\frac{1}{p}=\sum_{i=1}^m\frac{1}{p_i}.\label{Holder}
\end{align}
So it is natural to expect that $p_1,...,p_m$ satisfy this relationship for square function bounds of the form \eqref{Lpbound}.  For the remainder of this work, we will assume that $1<p_1,...,p_m<\infty$ and $p$ is defined by \eqref{Holder}.

When we are in the linear setting, with a convolution operator $\theta_t(x,y)=\psi_t(x-y)=t^{-n}\psi(t^{-1}(x-y))$, we use the notation \eqref{firstsqfunction} to avoid confusion with the square function $S$, and to emphasize that we are using the known Littlewood-Paley theory.

\begin{definition}\label{d:Apweight}
Let $w$ be a non-negative locally integrable function.  For $p>1$ we say that $w$ is an $A_p=A_p(\R^n)$ weight, written $w\in A_p$, if
\begin{align*}
[w]_{A_p}=\sup_Q\(\frac{1}{|Q|}\int_Qw(x)dx\)\(\frac{1}{|Q|}\int_Qw(x)^{1-p'}dx\)^{p-1}<\infty
\end{align*}
where the supremum is taken over all cubes $Q\subset\R^n$ with side parallel to the coordinate axes.
\end{definition}

The following lemma states that approximation to the identity operators have essentially the same convergence properties in weighted $L^p$ spaces as unweighted.  This result is well-known (an explicit proof is available for example in the work of Wilson \cite{W}), but for the reader's convenience we state the results precisely and give a short proof.

\begin{lemma}\label{l:conv}
Let $P_tf=\varphi_t*f$ where $|\varphi(x)|\less\frac{1}{(1+|x|)^N}$ for some $N>n$ with $\widehat\varphi(0)=1$ and $w\in A_p$ for some $1<p<\infty$.
\begin{itemize}
\item[\it i.] If $f\in L^p(w)$, then $P_tf\rightarrow f$ in $L^p(w)$ as $t\rightarrow0$.

\item[\it ii.] If $f\in L^p(w)$ and there exists a $1\leq q<\infty$ such that $f\in L^q$, then $P_tf\rightarrow 0$ in $L^p(w)$ as $t\rightarrow\infty$.
\end{itemize}
\end{lemma}

\begin{proof}
We first prove {\it (i)} by estimating
\begin{align*}
||P_tf-f||_{L^p(w)}&\leq\int_{\R^n}|\varphi(y)|\;||f(\cdot-ty)-f(\cdot)||_{L^p(w)}dy.
\end{align*}
The integrand $|\varphi(y)|\;||f(\cdot-ty)-f(\cdot)||_{L^p(w)}$ is controlled by $2||f||_{L^p(w)}|\varphi(y)|$ which is an integrable function.  So by dominated convergence
\begin{align*}
\lim_{t\rightarrow0}||P_tf-f||_{L^p(w)}&\leq\int_{\R^n}|\varphi(y)|\lim_{t\rightarrow0}||f(\cdot-ty)-f(\cdot)||_{L^p(w)}dy=0.
\end{align*}
Therefore {\it (i)} holds.  Now for {\it (ii)}, suppose that $f\in L^p(w)\cap L^q(\R^n)$ for some $1\leq q<\infty$.  Then it follows that for all $x\in\R^n$
\begin{align*}
|P_tf(x)|&\leq||\varphi_t||_{L^{q'}}||f||_{L^q}\\
&\less t^{-n/q}\(\int_{\R^n}\frac{dx}{(1+|x|)^{Nq'}}\)^{1/q'}||f||_{L^q}\\
&\less t^{-n/q}||f||_{L^q}
\end{align*}
which tends to $0$ as $t\rightarrow\infty$.  So $P_tf\rightarrow0$ a.e. in $\R^n$.   Furthermore $|P_tf(x)|\less Mf(x)$ where $Mf\in L^p(w)$ since $f\in L^p(w)$ and $1<p<\infty$.  Then by dominated convergence, we have
\begin{align*}
\lim_{t\rightarrow\infty}\int_{\R^n}|P_tf(x)|^pw(x)dx&=\int_{\R^n}\lim_{t\rightarrow\infty}|P_tf(x)|^pw(x)dx=0.
\end{align*}
So we have $P_tf\rightarrow0$ in $L^p(w)$ as $t\rightarrow\infty$.
\end{proof}

\begin{lemma}\label{l:decomp}
Suppose $\theta_t$ satisfies \eqref{size}, $P_tf=\varphi_t*f$ where $\varphi\in C_0^\infty$ with $\widehat\varphi(0)=1$, $w_i^{p_i}\in A_{p_i}$ for $1<p,p_1,...,p_m<\infty$ satisfying \eqref{Holder}.  Define $w=w_1\cdots w_m$.  Then for $f_i\in L^{p_i}(w_i^{p_i})\cap L^{p_i}$
\begin{align}
\Theta_t(f_1,...,f_m)=\sum_{j=1}^m\int_0^\infty\Theta_t\Pi_{j,s}(f_1,...,f_m)\frac{ds}{s}\label{decomp}
\end{align}
where the convergence holds in $L^p(w^p)$ and for $j=1,...,m$, $\Pi_{j,s}$ is defined by
\begin{align*}
&\Pi_{j,t}(f_1,...,f_m)=P_t^2f_1\otimes\cdots\otimes P_t^2f_{j-1}\otimes Q_tf_j\otimes P_t^2f_{j+1}\otimes\cdots\otimes P_t^2f_m,
\end{align*}
$Q_tf=\psi_t*f$, and $\psi_t=-t\frac{d}{dt}(\varphi_t*\varphi_t)$.  Furthermore there exist $Q_t^{i,k}f=\psi_t^{i,k}*f$ where $\psi^{i,k}\in C_0^\infty$ have mean zero for $i=1,2$ and $k=1,...,n$ and
\begin{align*}
Q_t=\sum_{k=1}^nQ_t^{1,k}Q_t^{2,k}.
\end{align*}
\end{lemma}

\begin{proof}
We note that since $f_i\in L^{p_i}(w_i^{p_i})\cap L^{p_i}$, by Lemma \ref{l:conv} $P_t^2f_i\rightarrow f_i$ as $t\rightarrow0$ and $P_t^2f_i\rightarrow0$ as $t\rightarrow\infty$ in $L^{p_i}(w_i^{p_i})$.  Then it follows that
\begin{align*}
&\left|\left|\Theta_t(f_1,...,f_m)-\sum_{j=1}^m\int_\epsilon^{1/\epsilon}\Theta_t\Pi_{j,s}(f_1,...,f_m)\frac{ds}{s}\right|\right|_{L^p(w^p)}\\
&\hspace{.5cm}=\left|\left|\Theta_t(f_1,...,f_m)+\int_\epsilon^{1/\epsilon}s\frac{d}{ds}\Theta_t(P_s^2 f_1,...,P_s^2f_m)\frac{ds}{s}\right|\right|_{L^p(w^p)}\\
&\hspace{.5cm}\leq\left|\left|\Theta_t(f_1,...,f_m)-\Theta_t(P_\epsilon^2 f_1,...,P_\epsilon^2f_m)\right|\right|_{L^p}+||\Theta_t(P_{1/\epsilon}^2 f_1,...,P_{1/\epsilon}^2f_m)||_{L^p(w^p)}\\
&\hspace{.5cm}\leq\sum_{j=1}^m\left|\left|\Theta_t(P_\epsilon^2 f_1,...,P_\epsilon f_{j-1},f_j-P_\epsilon^2 f_j,f_{j+1},...,f_m)\right|\right|_{L^p(w^p)}+||\Theta_t(P_{1/\epsilon}^2 f_1,...,P_{1/\epsilon}^2f_m)||_{L^p(w^p)}\\
&\hspace{.5cm}\less\sum_{j=1}^m\left|\left|Mf_1\cdots Mf_{j-1}(f_j-P_\epsilon f_j)f_{j+1}\cdots f_m\right|\right|_{L^p(w^p)}+||MP_{1/\epsilon}^2 f_1\cdots MP_{1/\epsilon}^2f_m||_{L^p(w^p)}\\
&\hspace{.5cm}\less\sum_{j=1}^m||f_j-P_\epsilon f_j||_{L^{p_j}(w_j^{p_j})}\prod_{i\neq j}||f_i||_{L^{p_i}(w_i^{p_i})}+\prod_{i=1}^m||P_{1/\epsilon}^2 f_i||_{L^{p_i}(w_i^{p_i})}.
\end{align*}
As $\epsilon\rightarrow0$, the above expression tends to zero.  Therefore we have \eqref{decomp} where the convergence is in $L^p(w^p)$.  One can verify that $\psi^{1,k}(x)=-2\partial_{x_k}\varphi(x)$ and $\psi^{2,k}(x)=x_k\varphi(x)$ satisfy the conditions given above.  For details, this decomposition of $Q_t$ was done in the linear one dimensional case by Coifman-Meyer in \cite{CM} and in the $n$ dimensional case by Grafakos in \cite{G}.
\end{proof}

\begin{lemma}\label{l:almostorth}
Let $P_t$, $Q_t$, $Q_t^{i,j}$, $\Pi_{j,s}$ be as in Lemma \ref{decomp}.  Then for all $f_i\in L^{p_i}(w_i^{p_i})\cap L^\infty_c$, $s>0$, $j=1,...,m$ and $x\in\R^n$
\begin{align*}
|\Theta_t\Pi_{j,s}(f_1,...,f_m)(x)|\less\(\frac{s}{t}\wedge\frac{t}{s}\)^{\gamma\,'}\sum_{k=1}^nMQ_s^{2,k}f_j(x)\prod_{i\neq j}Mf_i(x)
\end{align*}
for some $0<\gamma\,'\leq\gamma$ where $u\wedge v=\min(u,v)$ for $u,v>0$.
\end{lemma}

This lemma is a pointwise result that was proved in the discrete bilinear setting in \cite{Hart1}.  We make the appropriate modifications here to prove this multilinear continuous version.

\begin{proof}
For this proof, we define for $M,t>0$ and $x\in\R^n$
\begin{align}
\Phi_t^M(x)=\frac{t^{-n}}{(1+t^{-1}|x|)^M}.\label{bigphi}
\end{align}
It follows immediately that $\Phi_t^{M+d}\leq\Phi_t^M$ for any $d\geq0$, and there is a well known almost orthogonality result, for any $M,L>n$ and $s,t>0$
\begin{align}
\int_{\R^n}\Phi_t^M(x-u)\Phi_s^L(u-y)du\less\Phi_s^{M\wedge L}(x-y)+\Phi_t^{M\wedge L}(x-y).\label{aoestimate}
\end{align}
Note also that if we take $\eta=\frac{N-n}{2(N+\gamma)}$, $\gamma\,'=\eta\gamma$, and $N'=(1-\eta)N-\gamma\,'$, then using a geometric mean with weights $1-\eta$ and $\eta$ of estimates \eqref{size} and \eqref{regy} it follows that
\begin{align*}
|\theta_t(x,y_1,...,y_m)-\theta_t(x,y_1',y_2,...,y_m)|&\less t^{-\eta mn}(t^{-1}|y_1-y_1'|)^{\eta\gamma}\(\prod_{j=2}^m\Phi_t^N(x-y_j)\)^{1-\eta}\\
&\hspace{2.25cm}\times\(\Phi_t^N(x-y_1)+\Phi_t^N(x-y_1')\)^{1-\eta}\\
&\hspace{-3cm}=(t^{-1}|y_1-y_1'|)^{\gamma\,'}\(\Phi_t^{N'+\gamma\,'}(x-y_1)+\Phi_t^{N'+\gamma\,'}(x-y_1')\)\prod_{j=2}^m\Phi_t^{N'+\gamma\,'}(x-y_j)
\end{align*}
It is a direct computation to show that $0<\gamma\,'=\gamma\frac{N-n}{2(N+\gamma)}<\gamma$ and $n<N'=\frac{N+n}{2}\leq N-\gamma'$.  We will first look at the kernel of $\Theta_t(Q_s^{1,k}\,\cdot,P_s\;\cdot,...,P_s\;\cdot)$ for $k=1,...,m$, which is
\begin{align*}
\sum_{k=1}^n\int_{\R^{mn}}\theta_t(x,u_1,...,u_m)\psi_s^{1,k}(u_1-y_1)\prod_{i=2}^m\varphi_s(u_i-y_i)d\vec u.
\end{align*}
The goal here is to bound this kernel by a product of $\Phi_s^{N'}(x-y_j)+\Phi_t^{N'}(x-y_j)$.  So in the following computations, whenever possible we pull out terms of the form $\Phi_s^{N'}(x-y_j)$.  There will also appear terms of the form $\Phi_t^{N'}(x-u_j)$ and $\Phi_s^{N'}(u-y_j)$, for which we will use \eqref{aoestimate} and bound by appropriate functions $\Phi$ depending on $s$, $t$, and $x-y_j$.  We estimate the kernel for a fixed $k=1,...,m$ and simplify notation 
\begin{align*}
\lambda_s(y_1,...,y_m)=\psi_s^{1,k}(y_1)\prod_{i=2}^m\varphi_s(y_i).
\end{align*}
Then for $s<t$, it follows using that $\lambda_s(y_1,...,y_m)$ has mean zero in $y_1$ (since $\psi_s^{1,k}$ has mean zero), $\psi^{1,k},\varphi\in C_0^\infty$, and $\theta_t$ satisfies \eqref{size} and \eqref{regy} that
\begin{align}
&\left|\int_{\R^{mn}}\theta_t(x,u_1,...,u_m)\lambda_s(u_1-y_1,...,u_m-y_m)d\vec u\right|\notag\\
&\hspace{.5cm}\less\int_{\R^{mn}}|\theta_t(x,u_1,...,u_m)-\theta_t(x,y_1,u_2,...,u_m)|\(\prod_{j=1}^m\Phi_s^{N'+\gamma\,'}(u_j-y_j)\)d\vec u\notag\\
&\hspace{.5cm}\less\int_{\R^{mn}}(t^{-1}|u_1-y_1|)^{\gamma\,'}\Phi_t^{N'+\gamma\,'}(x-y_1)\Phi_s^{N'+\gamma\,'}(u_1-y_1)\notag\\
&\hspace{4.5cm}\times\prod_{j=2}^m\(\Phi_t^{N'+\gamma\,'}(x-u_j)\Phi_s^{N'+\gamma\,'}(u_j-y_j)\)d\vec u\notag\\
&\hspace{2cm}+\int_{\R^{mn}}(t^{-1}|u_1-y_1|)^{\gamma\,'}\prod_{j=1}^m\(\Phi_t^{N'+\gamma\,'}(x-u_j)\Phi_s^{N'+\gamma\,'}(u_j-y_j)\)d\vec u\notag\\
&\hspace{.5cm}\leq\frac{s^{\gamma\,'}}{t^{\gamma\,'}}\Phi_t^{N'+\gamma\,'}(x-y_1)\int_{\R^{mn}}\Phi_s^{N'}(u_1-y_1)\prod_{j=2}^m\(\Phi_t^{N'+\gamma\,'}(x-u_j)\Phi_s^{N'+\gamma\,'}(u_j-y_j)\)d\vec u\notag\\
&\hspace{2cm}+\frac{s^{\gamma\,'}}{t^{\gamma\,'}}\int_{\R^{mn}}\prod_{j=1}^m\(\Phi_t^{N'+\gamma\,'}(x-u_j)\Phi_s^{N'}(u_j-y_j)\)d\vec u\notag\\
&\hspace{.5cm}\less\frac{s^{\gamma\,'}}{t^{\gamma\,'}}\prod_{j=1}^m\(\Phi_s^{N'}(x-y_j)+\Phi_t^{N'}(x-y_j)\).\label{est1}
\end{align}
Note that we use the computation $(t^{-1}|u_1-y_1|)^{\gamma\,'}\Phi_s^{N'+\gamma\,'}(u_1-y_1)\leq\frac{s^{\gamma\,'}}{t^{\gamma\,'}}\Phi_s^{N'}(u_1-y_1)$.  Now for $s>t$, we use the assumptions $\Theta_t(1,...,1)=0$, $\theta_t$ satisfies \eqref{size}, and that $\psi_s^{1,k},\varphi_s\in C_0^\infty$ for the following estimate
\begin{align}
&\left|\int_{\R^{mn}}\theta_t(x,u_1,...,u_m)\lambda_s(u_1-y_1,...,u_m-y_m)d\vec u\right|\label{kerest}\\
&\hspace{.5cm}\less\int_{\R^{mn}}\prod_{j=1}^m\Phi_t^{N'+\gamma\,'}(x-u_j)\left|\lambda_s(u_1-y_1,...,u_m-y_m)-\lambda_s(x-y_1,...,x-y_m)\right|d\vec u\notag
\end{align}
Next we work to control the second term in the integrand on the right hand side of \eqref{kerest}.  Adding and subtracting successive terms, we get
\begin{align*}
&\left|\lambda_s(u_1-y_1,...,u_m-y_m)-\lambda_s(x-y_1,...,x-y_m)\right|\\
&\hspace{.5cm}\leq\sum_{\ell=1}^m|\lambda_s(u_1-y_1,...,u_{\ell-1}-y_{\ell-1},x-y_\ell,...,x-y_m)\\
&\hspace{5cm}-\lambda_s(u_1-y_1,...,u_\ell-y_\ell,x-y_{\ell+1},...,x-y_m)|\\
&\hspace{.5cm}\less\sum_{\ell=1}^m(s^{-1}|x-u_\ell|)^{\gamma'}\(\prod_{r=1}^{\ell-1}\Phi_s^{N'+\gamma'}(u_r-y_r)\)\(\Phi_s^{N'+\gamma'}(u_\ell-y_\ell)+\Phi_s^{N'+\gamma'}(x-y_\ell)\)\\
&\hspace{8.5cm}\times\(\prod_{r=\ell+1}^m\Phi_s^{N'+\gamma'}(x-y_r)\)\\
\end{align*}
Here we use the convection that $\prod_{j=1}^0A_j=\prod_{j=m+1}^mA_j=1$ to simplify notation.  Then \eqref{kerest} is bounded by
\begin{align}
&\sum_{\ell=1}^m\int_{\R^{mn}}\(\prod_{j=1}^m\Phi_t^{N'+\gamma'}(x-u_j)\)(s^{-1}|x-u_\ell|)^{\gamma'}\(\prod_{r=1}^{\ell-1}\Phi_s^{N'+\gamma'}(u_r-y_r)\)\notag\\
&\hspace{3cm}\times\(\Phi_s^{N'+\gamma'}(u_\ell-y_\ell)+\Phi_s^{N'+\gamma'}(x-y_\ell)\)\(\prod_{r=\ell+1}^m\Phi_s^{N'+\gamma'}(x-y_r)\)d\vec u\notag\\
&\hspace{.5cm}\leq\frac{t^{\gamma'}}{s^{\gamma'}}\sum_{\ell=1}^m\int_{\R^{mn}}\(\prod_{j=1}^m\Phi_t^{N'}(x-u_j)\)\(\prod_{r=1}^{\ell-1}\Phi_s^{N'+\gamma'}(u_r-y_r)\)\notag\\
&\hspace{3cm}\times\(\Phi_s^{N'+\gamma'}(u_\ell-y_\ell)+\Phi_s^{N'+\gamma'}(x-y_\ell)\)\(\prod_{r=\ell+1}^m\Phi_s^{N'+\gamma'}(x-y_r)\)d\vec u\notag\\
&\hspace{.5cm}\leq\frac{t^{\gamma'}}{s^{\gamma'}}\sum_{\ell=1}^m\(\prod_{r=1}^{\ell-1}\int_{\R^n}\Phi_t^{N'}(x-u_r)\Phi_s^{N'+\gamma'}(u_r-y_r)du_r\)\notag\\
&\hspace{2.5cm}\times\(\int_{\R^n}\Phi_t^{N'}(x-u_\ell)\(\Phi_s^{N'+\gamma'}(u_\ell-y_\ell)+\Phi_s^{N'+\gamma'}(x-y_\ell)\)du_\ell\)\notag\\
&\hspace{4.75cm}\times\(\prod_{r=\ell+1}^m\int_{\R^n}\Phi_t^{N'}(x-u_r)\Phi_s^{N'+\gamma'}(x-y_r)du_r\)\notag\\
&\hspace{.5cm}\leq\frac{t^{\gamma'}}{s^{\gamma'}}\prod_{r=1}^m\(\Phi_t^{N'}(x-y_r)+\Phi_t^{N'}(x-y_r)\).\label{est2}
\end{align}
Then using \eqref{est1} and \eqref{est2}, it follows that
\begin{align*}
&\left|\int_{\R^{mn}}\theta_t(x,u_1,...,u_m)\psi_s^{1,k}(u_1-y_1)\prod_{i=2}^m\varphi_s(u_i-y_i)d\vec u\right|\\
&\hspace{5cm}\less\(\frac{s}{t}\wedge\frac{t}{s}\)^{\gamma\,'}\prod_{j=1}^m\(\Phi_s^{N'}(x-y_j)+\Phi_t^{N'}(x-y_j)\).
\end{align*}
Then since $|\Phi_t^{N'}*f(x)|\less Mf(x)$ uniformly in $t$ and $\Theta_t\Pi_{s,1}=\sum_{k=1}^n\Theta(Q_s^{1,k}Q_s^{2,k},P_s^2,...,P_s^2)$, it follows that
\begin{align*}
|\Theta_t\Pi_{s,1}(f_1,...,f_m)(x)|\less\(\frac{s}{t}\wedge\frac{t}{s}\)^{\gamma\,'}\sum_{k=1}^nMQ_s^{2,k}f_1(x)\prod_{j=2}^mMf_j(x).
\end{align*}
By symmetry, this completes the proof.
\end{proof}

Next we work to set the square function results of \cite{Hart1}, \cite{GO} and \cite{GLMY} in weighted Lebesgue spaces.  This is a type of reduced T(1) Theorem for $L^2(\R_+,\frac{dt}{t})$-valued singular integral operators, where we assume that $\Theta_t(1,...,1)=0$ for all $t>0$.  We now state and prove a reduced T(1) Theorem for square functions on weighted spaces.

\begin{theorem}\label{t:weightedT1}
Let $\Theta_t$ and $S$ be defined as in \eqref{theta} and \eqref{sqfunction} where $\theta_t$ satisfies \eqref{size} and \eqref{regy}.  If $\Theta_t(1,...,1)=0$ for all $t>0$, then $S$ satisfies \eqref{Lpbound} for all $w_i^{p_i}\in A_{p_i}$, $1<p,p_1,...,p_m<\infty$ satisfying \eqref{Holder}, where $w=\prod_{i=1}^mw_i$, and $f_i\in L^{p_i}(w_i^{p_i})\cap L^{p_i}$.  Furthermore, the constant for this bound is at most a constant independent of $w_1,...,w_m$ times
\begin{align*}
\prod_{i=1}^m\(1+[w_j^{p_j}]_{A_{p_j}}^{\max(1,p_j'/p_j)+\max(\frac{1}{2},p_j'/p_j)}\).
\end{align*}
\end{theorem}

\begin{proof}
Let $P_t$, $Q_t$, etc. be defined as in Lemma \ref{l:decomp}, $f_i\in L^{p_i}(w_i^{p_i})\cap L^{p_i}$ and $h_t\in L^\infty_c$ for all $t>0$ such that
\begin{align*}
\left|\left|\(\int_0^\infty|h_t|^2\frac{dt}{t}\)^\frac{1}{2}\right|\right|_{L^{p'}(w^p)}\leq1.
\end{align*}
Recall that the dual of $L^p(w^p)$ can be realized as $L^{p'}(w^p)$ if we take the the measure space to be $\R^n$ with measure $w(x)^pdx$.  We estimate \eqref{Lpbound} by duality making use of Lemmas \ref{l:decomp} and \ref{l:almostorth}
\begin{align*}
&\left|\int_{\R^n}\int_0^\infty\Theta_t(f_1,...,f_m)(x)h_t(x)\frac{dt}{t}w(x)^pdx\right|\\
&\hspace{1cm}=\left|\int_{\R^n}\int_0^\infty\sum_{j=1}^m\int_0^\infty\Theta_t\Pi_{j,s}(f_1,...,f_m)(x)w(x)h_t(x)w(x)^{p/p'}\frac{ds}{s}\frac{dt}{t}dx\right|\\
&\hspace{1cm}\leq\sum_{j=1}^m\left|\left|\(\int_{[0,\infty)^2}\(\frac{s}{t}\wedge\frac{t}{s}\)^{-\gamma\,'}|\Theta_t\Pi_{j,s}(f_1,...,f_m)(x)|^2\frac{ds}{s}\frac{dt}{t}\)^\frac{1}{2}\right|\right|_{L^p(w^p)}\\
&\hspace{5cm}\times\left|\left|\(\int_{[0,\infty)^2}\(\frac{s}{t}\wedge\frac{t}{s}\)^{\gamma\,'}|h_t|^2\frac{ds}{s}\frac{dt}{t}\)^\frac{1}{2}\right|\right|_{L^{p'}(w^p)}\\
&\hspace{1cm}\less\sum_{j=1}^m\sum_{k=1}^n\left|\left|\(\int_{[0,\infty)^2}\(\frac{s}{t}\wedge\frac{t}{s}\)^{\gamma\,'}\(MQ_s^{2,k}f_j\prod_{i\neq j}Mf_i\)^2\frac{dt}{t}\frac{ds}{s}\)^\frac{1}{2}\right|\right|_{L^p(w^p)}\\
&\hspace{1cm}\less\sum_{j=1}^m\sum_{k=1}^n\left|\left|\(\int_0^\infty\(MQ_s^{2,k}f_j\)^2\frac{ds}{s}\)^\frac{1}{2}\prod_{i\neq j}Mf_i\right|\right|_{L^p(w^p)}\\
&\hspace{1cm}\less\sum_{j=1}^m\sum_{k=1}^n[w_j^{p_j}]_{A_{p_j}}^{\max(\frac{1}{2},p_j'/p_j)}||g_{\psi^{2,k}}(f_j)||_{L^{p_j}(w_j^{p_j})}\prod_{i\neq j}||Mf_i||_{L^{p_i}(w_i^{p_i})}\\
&\hspace{1cm}\less\sum_{j=1}^m[w_j^{p_j}]_{A_{p_j}}^{\max(1,p_j'/p_j)+\max(\frac{1}{2},p_j'/p_j)}||f_j||_{L^{p_j}(w_j^{p_j})}\prod_{i\neq j}[w_i^{p_i}]_{A_{p_i}}^\frac{1}{p_i-1}||f_i||_{L^{p_i}(w_i^{p_i})}\\
&\hspace{1cm}\less\prod_{i=1}^m\(1+[w_j^{p_j}]_{A_{p_j}}^{\max(1,p_j'/p_j)+\max(\frac{1}{2},p_j'/p_j)}\)||f_i||_{L^{p_i}(w_i^{p_i})}.
\end{align*}

Here we have used the weighted bound for the Hardy-Littlewood maximal function, the Fefferman-Stein vector-valued maximal function bound proved originally by Anderson-John \cite{AJ} and proved with the sharp dependence on the weight constant by Cruz-Uribe-Martell-Perez \cite{CUMP}.  We also used the weighted square function estimate for $g_{\psi^{2,k}}$ for $k=1,...,m$ originally proved by Kurtz \cite{K} and proved with sharp dependence on the weight constant by Lerner in \cite{L}.
\end{proof}

Although we use sharp estimates to track the weight constant dependence, we are not claiming that this bound on $S$ is sharp.  In the above argument, once we have bounded the dual pairing by products of maximal functions and $g_\psi$ functions, the estimates may be sharp, but there is no evidence provided here that the estimates up to that point are sharp.  We track the constant so that we can explicitly apply the extrapolation theorem of Grafakos-Martell \cite{GM}.

\section{Carleson and Strong Carleson Measures}

This section is dedicated to defining the cancellation conditions that we will use for $\theta_t$, and proving some properties about them.  We start with a discussion to motivate these definitions and describe the role that they will play in the theory.

As discussed in the introduction, in the linear convolution operator setting with convolutions kernel $\psi_t$, if $g_\psi$ is bounded, then necessarily $\psi_t*1=0$ for all $t>0$.  So when working with the square function $g_\psi$ with $\psi_t(x)=t^{-n}\psi(t^{-1}x)$, it is not useful to consider Carleson measure type cancellation conditions like {\it (i)} from Theorem \ref{t:main}.  But if one does not require the convolution kernels $\psi_t$ to be the dilations of a single function $\psi$ or allows for the non-convolution operators, then mean zero is not a necessary condition for square function bounds.  From the classical theory of Carleson measures \cite{C}, we know that in the linear setting $S$ is bounded on $L^2$ if and only if $|\Theta_t(1)(x)|^2\frac{dt\,dx}{t}$ is a Carleson measure, although this may not in general be sufficient for $S$ to be bounded for all $1<p<\infty$.  We will define the strong Carleson condition for $\Theta_t$ and prove that it does imply bounds for all $1<p<\infty$.  There is a stronger notion of Carleson measure defined by Journ\'e in \cite{Jo} that is related to some of the Carleson conditions in this work.  We will discuss this in a little more depth in Section 4.

\begin{definition}
A positive measure $d\mu(x,t)$ on $\R^{n+1}_{+}=\{(x,t):x\in\R^n,\;t>0\}$  is a {\it Carleson} measure if
\begin{align}\label{Carleson}
    \|d\mu\|_{\mathcal{C}}=\sup_Q\frac{1}{|Q|}d\mu(T(Q))<\infty\, ,
\end{align}
where the supremum is taken over all cubes $Q\subset\R^n$, $|Q|$ denotes the Lebesgue measure of the cube $Q$,   $T(Q)=Q\times (0,\ell(Q)]$ denotes the \emph{Carleson box} over $Q$, and  $\ell(Q)$ is the side length of $Q$. \\
Suppose $\mu$ is a non-negative measure on $\R^{n+1}_+$ defined by 
\begin{align}
d\mu(x,t)=F(x,t)d\tau(t)dx\label{SCform}
\end{align}
for some $F\in L^1_{loc}(\R^{n+1}_+,d\tau(t)dx)$.  We say that $\mu$ is a {\it strong Carleson} measure if
\begin{align}
||\mu||_{\mathcal{SC}}=\sup_Q\sup_{x\in Q}\int_0^{\ell(Q)}F(x,t)d\tau(t)<\infty.\label{strongCarleson}
\end{align}
Given an operator $\Theta_t$ with kernel satisfying \eqref{size}, we say that $\Theta_t$ satisfies the Carleson condition, respectively strong Carleson condition, if $|\Theta_t(1,...,1)(x)|^2\frac{dt}{t}dx$ is a Carleson measure, respectively strong Carleson measure.
\end{definition}

In \cite{CJ} and \cite{A}, Christ-Journ\'e and Auscher define a Carleson function to be a function $G:\R^{n+1}_+\rightarrow\C$ such that $|G(x,t)|^2\frac{dt}{t}dx$ is a Carleson measure.  So our definition of the Carleson condition for $\Theta_t$ is exactly that $G(x,t)=\Theta_t(1,...,1)(x)$ is a Carleson function in the language of Christ-Journ\'e and Auscher.  We state this definition with a general measure $d\tau(t)$ instead of just $\frac{dt}{t}$ because the results in Section 4 can be applied to the discrete case where $d\tau(t)=\delta_{2^{-k}}(t)$, like the ones in \cite{DRdF}, \cite{MN}, \cite{Hart1}, \cite{GLMY}, and many others.

It is trivial to see that if a non-negative measure $d\mu(x,t)=F(x,t)d\tau(t)dx$ is a strong Carleson measure, then it is a Carleson measure and $||\mu||_{\mathcal C}\leq||\mu||_{\mathcal{SC}}$, but we can also prove a partial converse to this for non-negative measures of the form $|\Theta_t(1,...,1)|^2\frac{dt\,dx}{t}$ for $\theta_t$ satisfying \eqref{size} and \eqref{regy}.  In Propositions \ref{p:CarlesonstrongCarleson} and \ref{p:Carlesontwocube}, we prove that $\Theta_t$ satisfies the two-cube and the Carleson conditions if and only if it satisfies the strong Carleson condition.  We first define the two-cube testing condition.

\begin{definition}
Let $\theta_t$ satisfy \eqref{size} and $\Theta_t$ be defined as in \eqref{theta}.  We say that $\Theta_t$ satisfies the {\it two-cube testing condition} if
\begin{align}
\sup_{R\subset Q}\frac{1}{|R|}\int_R\int_{\ell(R)}^{\ell(Q)}|\Theta_t(\chi_{(2R)^c},...,\chi_{(2R)^c})(x)-\Theta_t(\chi_{(2Q)^c},...,\chi_{(2Q)^c})(x)|^2\frac{dt}{t}dx<\infty,\label{twocube}
\end{align}
where the supremum is taken over all cubes $R$ and $Q$ with $R\subset Q$.
\end{definition}

In the linear case, the two-cube condition for $\Theta_t$ becomes
\begin{align*}
\sup_{R\subset Q}\frac{1}{|R|}\int_R\int_{\ell(R)}^{\ell(Q)}|\Theta_t(\chi_{2Q\backslash 2R})(x)|^2\frac{dt}{t}dx<\infty.
\end{align*}
The two-cube testing condition is a technical condition that arrises to conclude the uniform strong Carleson bound from the average control of the Carleson condition.  Before we verify the equivalence between these conditions, we first prove a lemma.

\begin{lemma}\label{l:testimates}
Suppose $\theta_t$ satisfies \eqref{size}.  Then we have the following
\begin{itemize}
\item[\it i.]  Suppose $E_1,...,E_m\subset\R^n$, then 
\begin{align}
\sup_{x\in\R^n}|\Theta_t(\chi_{E_1},...,\chi_{E_m})(x)|\less t^{-n}\min(|E_1|,...,|E_m|).\label{tlarge}
\end{align}
\item[\it ii.]  Suppose $E_1,...,E_m\subset\R^n$ and $2Q\subset\R^n\backslash E_i$ for some $i$ and cube $Q$ (here $2Q$ is the double of $Q$ with the same center), then
\begin{align}
\sup_{x\in Q}|\Theta_t(\chi_{E_1},...,\chi_{E_m})(x)|\less t^{N-n}\ell(Q)^{-(N-n)}\label{tsmall}
\end{align}
\end{itemize}
\end{lemma}

\begin{proof}
For $E_1,...,E_m\subset\R^n$ and $x\in\R^n$, using \eqref{size} we have
\begin{align*}
|\Theta_t(\chi_{E_1},...,\chi_{E_m})(x)|&\less\prod_{j=1}^m\int_{\R^n}\frac{t^{-n}}{(1+t^{-1}|x-y_j|)^N}\chi_{E_j}(y_j)dy_j\less t^{-n}|E_i|
\end{align*}
for each $i=1,...,m$.  For {\it (ii)}, for $x\in Q\subset2Q\subset\R^n\backslash E_i$, it follows that $|x-y_i|>\ell(Q)$ for all $y_i\in E_i$.  Then using \eqref{size}, it follows that
\begin{align*}
|\Theta_t(\chi_{E_1},...,\chi_{E_m})(x)|&\less\prod_{j=1}^m\int_{\R^n}\frac{t^{-n}}{(1+t^{-1}|x-y_j|)^N}\chi_{E_j}(y_j)dy_j\\
&\less\int_{E_i}\frac{t^{-n}}{(t^{-1}|x-y_i|)^N}dy_i\\
&\less t^{N-n}\int_{|x-y_i|>\ell(Q)}\frac{1}{|x-y_i|^N}dy_i\\
&\less t^{N-n}\ell(Q)^{-(N-n)}.
\end{align*}
\end{proof}

\begin{proposition}\label{p:CarlesonstrongCarleson}
Suppose $\theta_t$ satisfies \eqref{size} and \eqref{regy}.  If $\Theta_t(x)$ satisfies the Carleson and the two cube testing conditions, then $\Theta_t$ satisfies the strong Carleson condition.
\end{proposition}

\begin{proof}
We first prove a multilinear result analog of the result of Carleson and Christ-Journ\'e mentioned above, that $\Theta_t$ satisfies the Carleson condition implies that $S$ satisfies the unweighted bound \eqref{Lpbound} for $p=2$.  That is $d\mu(x,t)=|\Theta_t(1,...,1)(x)|^2\frac{dt\,dx}{t}$ is a Carleson measure implies for all $1<p_1,...,p_m<\infty$ satisfying \eqref{Holder} with $p=2$, $S$ is bounded from $L^{p_1}\times\cdots\times L^{p_m}$ into $L^2$.  To prove this we adapt a familiar technique from Coifman-Meyer, see e.g. \cite{CM2} or \cite{CM3}.  Decompose $\Theta_t=(\Theta_t-M_{\Theta_t(1,...,1)}\mathbb P_t)-M_{\Theta_t(1,...,1)}\mathbb P_t=R_t+U_t$ where 
\begin{align}
\mathbb P_t(f_1,...,f_m)=\prod_{i=1}^mP_tf_i\label{appid}
\end{align}
and $P_t$ is a smooth approximation to the identity.  The operator $R_t$ satisfies the conditions of Theorem \ref{t:weightedT1}, and hence the square function associated to $R_t$ is bounded on the appropriate spaces.  The second term is bounded as well using the following Carleson measure bound
\begin{align*}
\left|\left|\(\int_0^\infty|U_t(f_1,...,f_m)|^2\frac{dt}{t}\)^\frac{1}{2}\right|\right|_{L^2}&\leq\prod_{i=1}^m\(\int_{\R^{n+1}_+}|P_tf_i(x)|^{p_i}d\mu(x,t)\)^\frac{1}{p_i}\\
&\less\prod_{i=1}^m||f_i||_{L^{p_i}}.
\end{align*}
We use a bound proved by Carleson \cite{C}, that $\{P_t\}_{t>0}$ is bounded from $L^q(\R^n)$ into $L^q(\R^{n+1}_+,d\mu)$ for all $1<q<\infty$ whenever $d\mu(x,t)$ is a Carleson measure.  We now move on to estimate \eqref{strongCarleson}, so take a cube $Q\subset\R^n$ and define
\begin{align*}
G_Q(x)=\chi_Q(x)\int_0^{\ell(Q)}d\mu(x,t).
\end{align*}
To prove that $\mu$ is a strong Carleson measure, it is sufficient to show that $||G_Q||_{L^\infty}\less1$ where the constant is independent of $Q\subset\R^n$.  Since $d\mu$ is locally integrable in $\R^{n+1}_+$ and $d\mu$ is a Carleson measure, it follows that $G_Q\in L^1(\R^n)$.  Then we have that $G_Q(x)\leq MG_Q(x)$ for almost every $x\in\R^n$.  So we estimate $||MG_Q||_{L^\infty}$
\begin{align*}
MG_Q(x)&=\sup_{R\ni x}\frac{1}{|R|}\int_R\int_0^{\ell(Q)}|\Theta_t(1,...,1)(y)|^2\chi_Q(y)\frac{dt}{t}dy\\
&=\sup_{R\ni x:\;R\subset Q}\frac{1}{|R|}\int_R\int_0^{\ell(Q)}|\Theta_t(1,...,1)(y)|^2\frac{dt}{t}dy\\
&\leq\sup_{R\ni x:\;R\subset Q}\frac{1}{|R|}\int_R\int_0^{\ell(Q)}|\Theta_t(\chi_{2R},...,\chi_{2R})(y)|^2\frac{dt}{t}dy\\
&\hspace{1.5cm}+\sup_{R\ni x:\;R\subset Q}\sum_{\vec F\in\Lambda}\frac{1}{|R|}\int_R\int_0^{\ell(R)}|\Theta_t(\chi_{F_1},...,\chi_{F_m})(y)|^2\frac{dt}{t}dy\\
&\hspace{2.5cm}+\sup_{R\ni x:\;R\subset Q}\sum_{\vec F\in\Lambda}\frac{1}{|R|}\int_R\int_{\ell(R)}^{\ell(Q)}|\Theta_t(\chi_{F_1},...,\chi_{F_m})(y)|^2\frac{dt}{t}dy\\
&=I+II+III.
\end{align*}
where
\begin{align*}
\Lambda=\{\vec F=(F_1,...,F_m):F_i=2R\text{ or }F_i=(2R)^c\}\backslash\{(2R,...,2R)\}.
\end{align*}
Note that we may make the reduction to cubes $R\subset Q$ since $\supp(G_Q)\subset Q$ and $G_Q\geq0$.  For each cube $R\subset Q\subset\R^n$, we estimate $I$ using that boundedness of $S$
\begin{align*}
\frac{1}{|R|}\int_R\int_0^{\ell(Q)}|\Theta_t(\chi_{2R},...,\chi_{2R})(y)|^2\chi_R(y)\frac{dt}{t}dy&\leq\frac{1}{|R|}\int_{\R^n}\int_0^\infty|\Theta_t(\chi_{2R},...,\chi_{2R})(y)|^2\frac{dt}{t}dy\\
&\less\frac{1}{|R|}\prod_{i=1}^m||\chi_{2R}||_{L^{p_i}}^2\less1.
\end{align*}
Therefore $I$ is bounded independent of $x$ and $Q$.  We bound the second term there exists at least one $F_i=(2R)^c$.  Then using \eqref{tsmall} from Lemma \ref{l:testimates}, we have
\begin{align*}
\frac{1}{|R|}\int_R\int_0^{\ell(R)}|\Theta_t(\chi_{F_1},...,\chi_{F_m})(y)|^2\frac{dt}{t}dy&\less\frac{1}{|R|}\int_R\int_0^{\ell(R)}\frac{t^{2(N-n)}}{\ell(R)^{2k(N-n)}}\frac{dt}{t}dy\less1.
\end{align*}
Since $|\Lambda|=2^m-1$, this is sufficient to bound $II$.  Now for the term $III$, we first take $\vec F\in\Lambda$ such that at least one component $F_i=2R$.  Then by \eqref{tlarge} from Lemma \ref{l:testimates} we have
\begin{align*}
\frac{1}{|R|}\int_R\int_{\ell(R)}^{\ell(Q)}|\Theta_t(\chi_{F_1},...,\chi_{F_m})(y)|^2\frac{dt}{t}dy&\less\frac{1}{|R|}\int_R\int_{\ell(R)}^\infty t^{-2n}|2R|^2\frac{dt}{t}dy\less1.
\end{align*}
This bounds all but one term for $III$.  It remains to bound the term where $\vec F=((2R)^c,...,$ $(2R)^c)$.  We do this using \eqref{tsmall} from Lemma \ref{l:testimates} and the two cube condition \eqref{twocube}
\begin{align*}
&\frac{1}{|R|}\int_R\int_{\ell(R)}^{\ell(Q)}|\Theta_t(\chi_{(2R)^c},...,\chi_{(2R)^c})(y)|^2\frac{dt}{t}dy\\
&\hspace{1cm}\leq\frac{1}{|R|}\int_R\int_{\ell(R)}^{\ell(Q)}|\Theta_t(\chi_{(2Q)^c},...,\chi_{(2Q)^c})(y)|^2\frac{dt}{t}dy\\
&\hspace{2cm}+\frac{1}{|R|}\int_R\int_{\ell(R)}^{\ell(Q)}|\Theta_t(\chi_{(2Q)^c},...,\chi_{(2Q)^c})(y)-\Theta_t(\chi_{(2R)^c},...,\chi_{(2R)^c})(y)|^2\frac{dt}{t}dy\\
&\hspace{1cm}\less\frac{1}{|R|}\int_R\int_0^{\ell(Q)}t^{2(N-n)}\ell(Q)^{-2(N-n)}\frac{dt}{t}dy+1\less1
\end{align*}
Therefore $||MG_Q||_{L^\infty}\leq I+II+III\less1$ for all $Q\subset\R^n$ where the constant is independent of $Q$.  Now we can verify that $d\mu$ satisfies the strong Carleson condition
\begin{align*}
\sup_{Q\subset\R^n}\sup_{x\in Q}\int_0^{\ell(Q)}|\Theta_t(1,...,1)(x)|^2\frac{dt}{t}&\leq\sup_{Q\subset\R^n}||G_Q||_{L^\infty}\leq\sup_{Q\subset\R^n}||MG_Q||_{L^\infty}\less1.
\end{align*}
This completes the proof.
\end{proof}

\begin{proposition}\label{p:Carlesontwocube}
If $\theta_t$ satisfies \eqref{size}, \eqref{regy} and $\Theta_t$ satisfies the strong Carleson condition, then $\Theta_t$ satisfies the two cube condition \eqref{twocube}.
\end{proposition}

\begin{proof}
We estimate \eqref{twocube} for $R\subset Q\subset\R^n$
\begin{align*}
&\frac{1}{|R|}\int_R\int_{\ell(R)}^{\ell(Q)}|\Theta_t(\chi_{(2R)^c},...,\chi_{(2R)^c})(x)-\Theta_t(\chi_{(2Q)^c},...,\chi_{(2Q)^c})(x)|^2\frac{dt}{t}dx\\
&\hspace{1cm}\leq\sum_{j=1}^m\frac{1}{|R|}\int_R\int_{\ell(R)}^{\ell(Q)}|\Theta_t(\chi_{(2R)^c},...,\chi_{(2R)^c}-\chi_{(2Q)^c},...,\chi_{(2Q)^c})(x)|^2\frac{dt}{t}dx\\
&\hspace{1cm}\leq\frac{1}{|R|}\int_R\int_{\ell(R)}^{\ell(Q)}|\Theta_t(\chi_{(2R)^c},...,\chi_{(2R)^c},\chi_{2Q\backslash2R})(x)|^2\frac{dt}{t}dx\\
&\hspace{2cm}+\sum_{j=1}^{m-1}\frac{1}{|R|}\int_R\int_0^{\ell(Q)}|\Theta_t(\chi_{(2R)^c},...,\chi_{2Q\backslash 2R},...,\chi_{(2Q)^c})(x)|^2\frac{dt}{t}dx\\
&\hspace{1cm}\leq\frac{1}{|R|}\int_R\int_{\ell(R)}^{\ell(Q)}|\Theta_t(1,...,1)(x)-\Theta_t(\chi_{(2R)^c},...,\chi_{(2R)^c},\chi_{2Q\backslash2R})(x)|^2\frac{dt}{t}dx\\
&\hspace{4.75cm}+\frac{1}{|R|}\int_R\int_{\ell(R)}^{\ell(Q)}|\Theta_t(1,...,1)(x)|^2\frac{dt}{t}dx\\
&\hspace{5.25cm}+\sum_{j=1}^{m-1}\frac{1}{|R|}\int_R\int_0^{\ell(Q)}t^{2(N-n)}\ell(Q)^{-2(N-n)}\frac{dt}{t}dx\\
&\hspace{1cm}\less\frac{1}{|R|}\int_R\int_{\ell(R)}^{\ell(Q)}|\Theta_t(1,...,1)(x)-\Theta_t(\chi_{(2R)^c},...,\chi_{(2R)^c},\chi_{2Q\backslash2R})(x)|^2\frac{dt}{t}dx+1.
\end{align*}
Here the middle term is bounded by the assumption that $|\Theta_t(1,...,1)(x)|^2\frac{dt}{t}dx$ is a strong Carleson measure.  Now we bound
\begin{align*}
&|\Theta_t(1,...,1)(x)-\Theta_t(\chi_{(2R)^c},...,\chi_{(2R)^c},\chi_{2Q\backslash2R})(x)|\\
&\hspace{1cm}\leq\sum_{j=1}^{m-1}|\Theta_t(\chi_{2R},...,\chi_{2R},1,...,1)(x)|+|\Theta_t(\chi_{(2R)^c},...,\chi_{(2R)^c},1-\chi_{2Q\backslash2R})(x)|\\
&\hspace{1cm}\less\sum_{j=1}^{m-1}t^{-n}|R|+|\Theta_t(\chi_{(2R)^c},...,\chi_{(2R)^c},1-\chi_{2Q\backslash2R})(x)|\\
&\hspace{1cm}\less t^{-n}|R|+|\Theta_t(\chi_{(2R)^c},...,\chi_{(2R)^c},\chi_{(2Q)^c})(x)|+|\Theta_t(\chi_{(2R)^c},...,\chi_{(2R)^c},\chi_{2R})(x)|\\
&\hspace{1cm}\less t^{-n}|R|+t^{N-n}\ell(Q)^{-(N-n)}.
\end{align*}
In the second to last line we bound the last term by $t^{-n}|R|$ and absorb it into the first term of the last line.  Therefore we have that
\begin{align*}
&\frac{1}{|R|}\int_R\int_{\ell(R)}^{\ell(Q)}|\Theta_t(1,...,1)(x)-\Theta_t(\chi_{(2R)^c},...,\chi_{(2R)^c},\chi_{2Q\backslash2R})(x)|^2\frac{dt}{t}dx\\
&\hspace{1cm}\less\frac{1}{|R|}\int_R\int_{\ell(R)}^\infty t^{-2n}|R|^2\frac{dt}{t}dx+\frac{1}{|R|}\int_R\int_0^{\ell(Q)}t^{2(N-n)}\ell(Q)^{-2(N-n)}\frac{dt}{t}dx\less1,
\end{align*}
and hence $\Theta_t$ satisfies the two cube condition \eqref{twocube}.
\end{proof}

We also prove that if $S$ is bounded from $L^{p_1}\times\cdots\times L^{p_m}$ into $L^p$ for some $1<p_1,...,p_m<\infty$ and $2\leq p<\infty$ satisfying \eqref{Holder}, then $\Theta_t$ satisfies the Carleson condition.  A partial converse to this was proved within the proof of Proposition \ref{p:CarlesonstrongCarleson}:  If $\Theta_t$ satisfies the Carleson condition, then $S$ is bounded from $L^{p_1}\times\cdots\times L^{p_m}$ into $L^2$ for all $1<p_1,...,p_m<\infty$.

\begin{proposition}\label{p:boundedimpliesCarleson}
Assume $\theta_t$ satisfies \eqref{size} and $S$ is bounded from $L^{p_1}\times\cdots\times L^{p_m}$ into $L^p$ for some $1<p_1,...,p_m<\infty$ and $2\leq p<\infty$ satisfying \eqref{Holder}.  Then it follows that $\Theta_t$ satisfies the Carleson condition.
\end{proposition}

\begin{proof}
Fix a cube $Q\subset\R^n$ and we estimate
\begin{align}
\frac{1}{|Q|}\int_Q\int_0^{\ell(Q)}|\Theta_t(1,...,1)(x)|^2\frac{dt}{t}dx&\leq\frac{1}{|Q|}\int_Q\int_0^{\ell(Q)}|\Theta_t(\chi_{2Q},...,\chi_{2Q})(x)|^2\frac{dt}{t}dx\notag\\
&\hspace{.5cm}+\sum_{\vec F\in\Lambda}\frac{1}{|Q|}\int_Q\int_0^{\ell(Q)}|\Theta_t(\chi_{F_1},...,\chi_{F_m})(x)|^2\frac{dt}{t}dx\notag\\
&=I+II\label{IandII}
\end{align}
where
\begin{align*}
\Lambda=\{\vec F=(F_1,...,F_m):F_i=2Q\text{ or }F_i=(2Q)^c\}\backslash\{(2Q,...,2Q)\}.
\end{align*}
For each cube $Q\subset\R^n$, we estimate $I$
\begin{align*}
\frac{1}{|Q|}\int_Q\int_0^{\ell(Q)}|\Theta_t(\chi_{2Q},...,\chi_{2Q})(x)|^2\frac{dt}{t}dx&\leq\frac{1}{|Q|}\int_QS(\chi_{2Q},...,\chi_{2Q})(x)^2dx\\
&\leq\(\frac{1}{|Q|}\int_{\R^n}S(\chi_{2Q},...,\chi_{2Q})(x)^pdx\)^\frac{2}{p}\\
&\less|Q|^{-2/p}\prod_{i=1}^m||\chi_{2Q}||_{L^{p_i}}^2\less1.
\end{align*}
Now for the second term $II$, we fix $\vec F\in\Lambda$, which has at least one component $F_i=(2Q)^c$.  Then by \eqref{tsmall} from Lemma \ref{l:testimates} we have
\begin{align*}
\frac{1}{|Q|}\int_Q\int_0^{\ell(Q)}|\Theta_t(\chi_{F_1},...,\chi_{F_m})(x)|^2\frac{dt}{t}dx&\less\frac{1}{|Q|}\int_Q\int_0^{\ell(Q)}t^{2(N-n)}\ell(Q)^{-2(N-n)}\frac{dt}{t}dx\less1.
\end{align*}
Now noting that $|\Lambda|=2^m-1$, it follows that $II\less1$ as well.  So $\Theta_t$ satisfies the Carleson condition.
\end{proof}

In fact, this proves that if $\theta_t$ satisfies \eqref{size}, \eqref{regy} and $\Theta_t$ satisfies the Carleson condition, then $\Theta_t$ satisfies the strong Carleson condition if and only if $\Theta_t$ satisfies the two cube testing condition \eqref{twocube}.  We conclude this section with a few examples of various Carleson measure obtained from operators $\Theta_t$ satisfying \eqref{size} and \eqref{regy}.  
In Example \ref{ex:2}, we define a operators that give rise to strong Carleson measures, and in Example \ref{ex:3}, we define operators that give rise to operators that are Carleson measures, but not strong Carleson measures.  For the examples, let $P_t$ be a smooth approximation to the identity and $\mathbb P_t$ be as defined in \eqref{appid}.


\begin{example}\label{ex:2}
Suppose $\psi\in L^1$ with integral zero satisfying $|\psi(x)|\less\frac{1}{(1+|x|)^N}$ 
\begin{align}
\sup_{\xi\neq0}\int_0^\infty|\widehat\psi(t\xi)|^2\frac{dt}{t}<\infty,\label{FTmeanzero}
\end{align}
and define $Q_tf=\psi_t*f$.  Let $b\in L^q$ for some $1\leq q<\infty$ with $|b(x)-b(x')|\leq L|x-x'|^\alpha$ where $0<\alpha<N-n$, $\beta\in L^\infty(\R^{n+1}_+)$, and define $D_t(f_1,...,f_m)(x)=\beta(x,t)Q_tb(x)\mathbb P_t(f_1,...,f_m)(x)$.  It follows that the kernels of $D_t$, which are for $t>0$ 
\begin{align*}
d_t(x,y_1,...,y_m)=\beta(x,t)Q_tb(x)\prod_{i=1}^m\varphi_t(x-y_i),
\end{align*}
satisfy \eqref{size} and \eqref{regy}.  We also have that $\Theta_t(1,...,1)=\beta(x,t)Q_tb$, so we estimate
\begin{align*}
|Q_tb(x)|=\left|\int_{\R^n}\psi_t(x-y)(b(y)-b(x))dy\right|&\leq L\int_{\R^n}|\psi_t(x-y)|\,|x-y|^\alpha dy\\
&\less t^\alpha\int_{\R^n}\frac{t^{-n}}{(1+t^{-1}|x-y|)^{N-\alpha}}dy\less t^\alpha.
\end{align*}
Also we have that
\begin{align*}
|Q_tb(x)|&\leq||\psi_t||_{L^{q'}}||b||_{L^q}\less t^{-n/q}.
\end{align*}
Then it follows that
\begin{align*}
\int_0^{\ell(Q)}|\Theta_t(1,...,1)(x)|^2\frac{dt}{t}&\less ||\beta||_{L^\infty(\R^{n+1}_+)}^2\int_0^1t^{2\alpha}\frac{dt}{t}+||\beta||_{L^\infty(\R^{n+1}_+)}^2\int_1^\infty t^{-2n/q}\frac{dt}{t}\less1.
\end{align*}
Therefore with this selection of $b$ and $\beta$, it follows that $D_t$ satisfies the strong Carleson condition.  So by Theorem \ref{t:main}, it follows that
\begin{align*}
\left|\left|\(\int_0^\infty|D_t(f_1,...,f_m)|^2\frac{dt}{t}\)^\frac{1}{2}\right|\right|_{L^p(w^p)}\less\prod_{i=1}^m||f_i||_{L^{p_i}(w_i^{p_i})}
\end{align*}
for all $1<p_1,...,p_m<\infty$ and $w_i^{p_i}\in A_{p_i}$ where $w=w_1\cdots w_m$ and $p$ is defined by \eqref{Holder}, which allows for $1/m<p<\infty$.  Note that with an appropriate selection of $\beta_t$, the kernels $d_t(x,y)$ will not be smooth in the $x$ variable.  This is an operator to which one could not apply previous results.  Even in the linear case, one needed smoothness in $x$ to conclude bounds for for $p>2$ from the Carleson condition on $\Theta_t$.
\end{example}

\begin{example}\label{ex:3}
The purpose of this example is to construct an operator $\Theta_t$ satisfying \eqref{size} and \eqref{regy} such that $\Theta_t$ satisfies the Carleson condition, but not the strong Carleson condition.  Define $\psi(x)=\chi_{(0,1)}(x)-\chi_{(-1,0)}(x)$, $Q_tf=\psi_t*f$, $b(x)=\chi_{(0,1)}(x)$, and like above $D_t(f_1,...,f_m)(x)=Q_tb(x)\mathbb P_t(f_1,...,f_m)(x)$.  As above, we have that $D_t(1,...,1)=Q_tb$.  It is a quick computation to show that
\begin{align*}
\widehat\psi(\xi)=2\frac{1-\cos(\xi)}{i\xi}
\end{align*}
with the appropriate modification when $\xi=0$.  It follows then that $|\widehat\psi(\xi)|\less\min(|\xi|,|\xi|^{-1})$, and that
\begin{align*}
|D_t(1,...1)(x)|^2\frac{dt}{t}dx=|\psi_t*b(x)|^2\frac{dt}{t}dx
\end{align*}
is a Carleson measure.  Now we show that $D_t$ does not satisfy the strong Carleson condition.  Let $Q=[-1,0]$, $x\in[-1,0)\subset Q$, and we estimate \eqref{strongCarleson} with the following computation
\begin{align*}
\int_0^{\ell(Q)}|D_t1(x)|^2\frac{dt}{t}&=\int_0^1\left|\int_\R\psi_t(y)\chi_{(0,1)}(x-y)dy\right|^2\frac{dt}{t}\\
&\geq\int_{-x}^1\left|\int_{-t}^x\psi_t(y)dy\right|^2\frac{dt}{t}\\
&=\int_{-x}^1\frac{(x+t)^2}{t^2}\frac{dt}{t}\\
&=x^2\int_{-x}^1\frac{dt}{t^3}+2x\int_{-x}^1\frac{dt}{t^2}+\int_{-x}^1\frac{dt}{t}\\
&\geq x^2\int_0^1dt-2x-2-\log(-x)\\
&\geq-\log(-x)-2.
\end{align*}
Therefore
\begin{align*}
\sup_{x\in[-1,0]}\int_0^{\ell(Q)}|D_t1(x)|^2\frac{dt}{t}&\geq\sup_{x\in[-1,0)}-\log(-x)-2=\infty,
\end{align*}
and hence $D_t$ satisfies the Carleson condition, but not the strong Carleson condition.
\end{example}

\section{A Full Weighted T1 Theorem for Square Functions for $L^2$}

In this section, we develop some classical Carleson measure results in a weighted setting with strong Carleson measures.  With these new tools, we can apply some familiar arguments to complete the proof of Theorems \ref{t:main} and \ref{t:mconv}.  More precisely, Lemmas \ref{l:tentbound}, \ref{l:weightedA2bound} and Proposition \ref{p:strongCarlesonweightbound} are weighted versions of results proved by Carleson in \cite{C} where we use assume strong Carleson in place of Carleson conditions.

\begin{lemma}\label{l:tentbound}
If $\mu$ is a strong Carleson measure, then for any locally integrable function $w\geq0$ and $E\subset\R^n$
\begin{align}
\mu_w(\widehat E)\leq||\mu||_{\mathcal{SC}}\;w(E)\label{Journemeasure}
\end{align}
where $d\mu_w(x,t)=w(x)d\mu(x,t)$ and $\widehat E=\{(x,t)\in\R^{n+1}_+:B(x,t)\subset E\}$.
\end{lemma}

In \cite{Jo}, Journ\'e says that $d\mu_w$ is a Carleson measure with respect to $w\in A_2$ if it satisfied \eqref{Journemeasure}.  He uses this definition to prove that measures that satisfy this estimate also verify weighted analogs of Carleson measure bounds.  In particular, Journ\'e proves 

\begin{proof}
Let $Q_j$ be the Calder\'on-Zygmund decomposition of $\chi_E$ at height $\frac{1}{2}$.  Then
\begin{align*}
E\subset\bigcup_jQ_j\text{ and }|E|\leq\sum_j|Q_j|\leq2|E|.
\end{align*}
Let $Q_j^*$ be the dyadic cube with double the side length of $Q_j$ containing $Q_j$ and take $(x,t)\in\widehat E$.  Since $B(x,t)\subset E$ and $Q_j^*\not\subset E$, it follows that $B(x,t)\subset B(x,3\sqrt n\ell(Q_j))$.  Then 
\begin{align*}
\widehat E\subset\bigcup_jQ_j\times(0,2\sqrt n\ell(Q_j)]
\end{align*}
Now $d\mu(x,t)=F(x,t)d\tau(t)dx$ for some non-negative $F\in L^1_{loc}(\R^{n+1}_+)$.  So using that $\mu$ is a strong Carleson measure, it follows that
\begin{align*}
\mu_w(\widehat E)&\leq\sum_j\mu_w((E\cap Q_j)\times(0,2\sqrt n\ell(Q_j)])\\
&=\sum_j\int_{E\cap Q_j}\int_0^{2\sqrt n\ell(Q_j)}F(x,t)d\tau(t)w(x)\chi_{Q_j}(x)dx\\
&\leq||\mu||_{\mathcal{SC}}\sum_j\int_{E\cap Q_j}w(x)dx\\
&\leq||\mu||_{\mathcal{SC}}\;w(E).
\end{align*}
In the last line, we use that $E\cap Q_j$ are disjoint.
\end{proof}

\begin{lemma}\label{l:weightedA2bound}
Suppose $d\mu(x,t)=F(x,t)d\tau(t)dx$ is a strong Carleson measure and $|\phi(x)|\less\frac{1}{(1+|x|)^N}$ for some $N>n$.  Then for all $w\in A_p$ for $1<p<\infty$,
\begin{align}
\(\int_{\R^{n+1}_+}|\phi_t*f(x)|^pw(x)d\mu(x,t)\)^\frac{1}{p}\less||\mu||_{\mathcal{SC}}^{1/p}[w]_{A_p}^{1/(p-1)}||f||_{L^p(w)}.
\end{align}
\end{lemma}

\begin{proof}
Define the non-tangential maximal function
\begin{align*}
M_\phi f(x)=\sup_{t>0}\sup_{|x-y|<t}|\phi_t*f(t)|.
\end{align*}
For $\lambda>0$, define 
\begin{align*}
&E_\lambda=\{x\in\R^n:M_\phi f(x)>\lambda\}\\
&\widehat E_\lambda=\{(x,t)\in\R^{n+1}_+:B(x,t)\subset E_\lambda\}.
\end{align*}
It follows from Lemma \ref{l:tentbound} that $\mu_w(\widehat E_\lambda)\leq||\mu||_{\mathcal{SC}}\;w(E_\lambda)$ where again $d\mu_w(x,t)=w(x)d\mu(x,t)$.  Therefore
\begin{align*}
\int_{\R^{n+1}_+}|\phi_t*f(x)|^pw(x)d\mu(x,t)&=p\int_0^\infty\lambda^p\mu_w(\{(x,t)\in\R^{n+1}_+:|\phi_t*f(x)|>\lambda\})\frac{d\lambda}{\lambda}\\
&\leq p\int_0^\infty\lambda^p\mu_w(\widehat E_\lambda)\frac{d\lambda}{\lambda}\\
&\leq p||\mu||_{\mathcal{SC}}\int_0^\infty\lambda^pw(E_\lambda)\frac{d\lambda}{\lambda}\\
&=||\mu||_{\mathcal{SC}}\int_{\R^n}M_\phi f(x)^pw(x)dx\\
&\less||\mu||_{\mathcal{SC}}[w]_{A_p}^{p/(p-1)}||f||_{L^p(w)}^p.
\end{align*}
Here we use as before that $|\phi_t*f(x)|\less Mf(x)$ and $||Mf||_{L^p(w)}\less[w]_{A_p}^{1/(p-1)}||f||_{L^p(w)}$.
\end{proof}

\begin{proposition}\label{p:strongCarlesonweightbound}
Suppose $\theta_t$ satisfies \eqref{size} and \eqref{regy}.  If $\Theta_t$ satisfies the strong Carleson condition, then $S$ is satisfies \eqref{Lpbound} for all $w_i^{p_i}\in A_{p_i}$ and $1<p_1,...,p_m<\infty$ satisfying \eqref{Holder} with $p=2$ where $w=w_1\cdots w_m$.  Furthermore, the constant for this bound is at most a constant independent of $w_1,...,w_m$ times
\begin{align}
C_{m,n,w_1,...,w_m,p_1,...,p_m}&=\prod_{i=1}^m\(1+[w_i^{p_i}]_{A_{p_i}}^{\max(1,p_i'/p_i)+\max(1/2,p_i'/p_i)}\)\label{fullconstant}\\
&\hspace{5cm}+||\mu||_{\mathcal{SC}}^{m/2}\prod_{i=1}^m[w_i^{p_i}]_{A_{p_i}}^{p_i'/p_i}.\notag
\end{align}
\end{proposition}

\begin{proof}
Define $R_t=\Theta_t-M_{\Theta_t(1,...,1)}\mathbb P_t$ and $U_t=M_{\Theta_t(1,...,1)}\mathbb P_t$.  Then $R_t$ satisfies \eqref{size}, \eqref{regy}, and in addition $R_t(1,...,1)=0$ for all $t>0$.  Then by Theorem \ref{t:weightedT1}, it follows that
\begin{align*}
\left|\left|\(\int_0^\infty|R_t(f_1,...,f_m)|^2\frac{dt}{t}\)^\frac{1}{2}\right|\right|_{L^p(w^p)}\less\prod_{i=1}^m||f_i||_{L^{p_i}(w_i^{p_i})}.
\end{align*}
Now we turn to the $U_t$ term.  For any $w_i^{p_i}\in A_{p_i}$ for $1<p_1,...,p_m<\infty$ satisfying \eqref{Holder} with $p=2$, take $d\mu(x,t)=|\Theta_t(1,...,1)|^2\frac{dt\,dx}{t}$ it follows that
\begin{align*}
\left|\left|\(\int_0^\infty|U_t(f_1,...,f_m)|^2\frac{dt}{t}\)^\frac{1}{2}\right|\right|_{L^2(w^2)}^2&=\int_{\R^{n+1}_+}\(\prod_{i=1}^m|P_tf_i(x)|w_i(x)\)^2d\mu(x,t)\\
&\hspace{-.5cm}\leq\prod_{i=1}^m\(\int_{\R^{n+1}_+}|P_tf_i(x)|^{p_i}w_i(x)^{p_i}d\mu(x,t)\)^\frac{2}{p_i}\\
&\hspace{-.5cm}\less||\mu||_{\mathcal{SC}}^m\prod_{i=1}^m[w_i^{p_i}]_{A_{p_i}}^{2/(p_i-1)}||f_i||_{L^{p_i}(w_i^{p_i})}^2.
\end{align*}
The final inequality holds by Lemma \ref{l:weightedA2bound}.  The first term in the constant \eqref{fullconstant} is from the bound of $R_t$ by Theorem \ref{t:weightedT1} and the second term is from the bound of $U_t$ above.
\end{proof}

These results almost complete the proof of Theorem \ref{t:main}, except for dealing with a density issue with $f_i\in L^{p_i}(w_i^{p_i})\cap L^{p_i}$ and applying weight extrapolation.  Propositions \ref{p:CarlesonstrongCarleson} and \ref{p:Carlesontwocube} verify the equivalence of {\it (i)} and {\it (ii)} from Theorem \ref{t:main}.  By Proposition \ref{p:CarlesonstrongCarleson}, {\it (i)} implies that $S$ satisfies \eqref{Lpbound} for all $w_i^{p_i}\in A_{p_i}$ with $1<p_1,...,p_m$ and $p=2$ for $f_i\in L^{p_i}(w_i^{p_i})\cap L^{p_i}$.  In order to conclude boundedness for all $L^{p_i}(w_i^{p_i})$, we make a short density argument in following and apply the extrapolation theorem of Grafakos-Martell \cite{GM} to complete the proof of Theorem \ref{t:main}.  We will use a lemma to prove this.

\begin{lemma}\label{l:Lpw}
If $w\in A_p$ and $1<p<\infty$, then $\frac{1}{(d+|x_0-\;\cdot\;|)^n}\in L^p(w)$ for any $x_0\in\R^n$ and $d>0$.
\end{lemma}

\begin{proof}
We start  by noting that for any $x\in\R^n$
\begin{align*}
M\chi_{B(x_0,d)}(x)&\geq\frac{1}{|B(x,|x-x_0|+d)|}\int_{B(x,|x-x_0|+d)}\chi_{B(0,d)}(x)dx\\
&=\frac{|\chi_{B(x_0,d)}(x)|}{|B(x,|x-x_0|+d)|}=\frac{d^n}{(d+|x-x_0|)^n}.
\end{align*}
Then it follows that
\begin{align*}
\(\int_{\R^n}\frac{1}{(d+|x-x_0|)^{np}}w(x)dx\)^\frac{1}{p}\leq d^{-n}||M\chi_{B(x_0,d)}||_{L^p(w)}\less ||\chi_{B(x_0,d)}||_{L^p(w)}<\infty.
\end{align*}
Here we use the Hardy-Littlewood maximal operator bound on $L^p(w)$ and that $w\in L^1_{loc}$.
\end{proof}

\begin{proof}
First we restrict to the case $p=2$ and take $f_i\in L^{p_i}(w_i^{p_i})$ and $f_{i,k}\in L^{p_i}(w_i^{p_i})\cap L^{p_i}$ with $f_{i,k}\rightarrow f_i$ in $L^{p_i}(w_i^{p_i})$ as $k\rightarrow\infty$.  It follows that $f_{1,k}\otimes\cdots\otimes f_{m,k}\rightarrow f_1\otimes\cdots\otimes f_m$ as $k\rightarrow\infty$ in the weighted product Lebesgue space $L^{p_1}(w_1^{p_1})\cdots L^{p_m}(w_m^{p_m})$.  For all $x\in\R^n$
\begin{align*}
&|\Theta_t(f_1,...,f_m)(x)-\Theta_t(f_{1,k},...,f_{m,k})(x)|\\
&\hspace{.25cm}\leq\int_{\R^{mn}}|\theta_t(x,y_1,...,y_m)|\;|f_1(y_1)\cdots f_m(y_m)-f_{1,k}(y_1)\cdots f_{m,k}(y_m)|d\vec y\\
&\hspace{.25cm}\leq \prod_{i=1}^mt^{N-n}\(\int_{\R^n}\frac{w_i(y_i)^{-p_i'}dy_i}{(t+|x-y_i|)^{p_i'N}}\)^\frac{1}{p_i'}||f_1\otimes\cdots\otimes f_m-f_{1,k}\otimes\cdots\otimes f_{m,k}||_{L^{p_1}(w_1^{p_1})\cdots L^{p_m}(w_1^{p_m})},
\end{align*}
which tends to zero as $k\rightarrow\infty$ almost everywhere since $w_i^{p_i}\in A_{p_i}$ implies that $w_i^{-p_i'}\in A_{p_i'}$ and so the first term is finite almost everywhere by Lemma \ref{l:Lpw}.  Therefore $\Theta_t(f_{1,k},...,f_{m,k})\rightarrow\Theta_t(f_1,...,f_m)$ pointwise as $k\rightarrow\infty$ a.e. $x\in\R^n$.  Then by Fatou's lemma we have that
\begin{align*}
||S(f_1,...,f_m)||_{L^2(w^2)}^2&=\int_{\R^n}\int_0^\infty\lim_{k\rightarrow\infty}|\Theta_t(f_{1,k},...,f_{m,k})(x)|^2\frac{dt}{t}w(x)^2dx\\
&\leq\liminf_{k\rightarrow\infty}\int_{\R^n}\int_0^\infty|\Theta_t(f_{1,k},...,f_{m,k})(x)|^2\frac{dt}{t}w(x)^2dx\\
&\leq C_{n,m,w_1,...,w_m,p_1,...,p_m}\liminf_{k\rightarrow\infty}\prod_{i=1}^m||f_{i,k}||_{L^{p_i}(w_i^{p_i})}^2\\
&=C_{n,m,w_1,...,w_m,p_1,...,p_m}\prod_{i=1}^m||f_i||_{L^{p_i}(w_i^{p_i})}^2
\end{align*}
Therefore $S$ satisfies \eqref{Lpbound} for all $1<p_1,...,p_m<\infty$ satisfying \eqref{Holder} with $p=2$, for all $w_i^{p_i}\in A_{p_i}$, and for all $f_i\in L^{p_i}(w_i^{p_i})$.  We complete the proof by applying the multilinear extrapolation theorem of Grafakos-Martel \cite{GM}, which we state now.

\begin{theorem}[Grafakos-Martell \cite{GM}]\label{t:GM}
Let $1\leq q_1,...,q_m<\infty$ and $1/m\leq q<\infty$ be fixed indices that satisfy \eqref{Holder} and $T$ be an operator defined on $L^{q_1}(w_1^{q_1})\times\cdots\times L^{q_m}(w_m^{q_m})$ for all tuples of weights $w_i^{q_i}\in A_{q_i}$.  We suppose that for all $B>1$, there is a constant $C_0=C_0(B) >0$ such that for all tuples of weights $w_i^{q_i}\in A_{q_i}$ with $[w_i^{q_i}]_{A_{q_i}}\leq B$ and all functions $f_i\in L^{q_i}(w_i^{q_i})$, $T$ satisfies
\begin{align*}
||T(f_1,...,f_m)||_{L^q(w^q)}\leq C_0\prod_{i=1}^m||f_i||_{L^{q_i}(w_i^{q_i})}.
\end{align*}
Then for all indices $1<p_1,...,p_m<\infty$ and $1/m<p<\infty$ that satisfy \eqref{Holder}, all $B>1$, and all weights $w_i^{p_i}\in A_{p_i}$ with $[w_i^{p_i}]_{A_{p_i}}<B$, there is a constant $C=C(B)$ such that for all $f_i\in L^{p_i}(w_i^{p_i})$
\begin{align*}
||T(f_1,...,f_m)||_{L^p(w^p)}\leq C\prod_{i=1}^m||f_i||_{L^{p_i}(w_i^{p_i})}.
\end{align*}
\end{theorem}

We may take, for example, $q_1=\cdots=q_m=2m$ and hence $q=2$.  Then we have just proved that for all $B>1$ and $w_i^{q_i}\in A_{q_i}$ with $[w_i^{q_i}]_{A_{q_i}}\leq B$ that
\begin{align*}
||S(f_1,...,f_m)||_{L^2(w^2)}&\leq C_{n,m,q_1,...,q_m}C_{m,n,p_1,...,p_m,w_1,...,w_m}\prod_{i=1}^m||f_i||_{L^{q_i}(w_i^{q_i})}
\end{align*}
where $C_{m,n,w_1,...,w_m,q_1,...,q_m}$ is defined in \eqref{fullconstant}.  Since $C_{m,n,w_1,...,w_m,q_1,...,q_m}$ is an increasing some of power functions of $[w_i^{q_i}]_{A_{q_i}}$, one can define $C_0(B)$ by replacing the weight constants with $B$ in \eqref{fullconstant} times a constant independent of the weights,
\begin{align*}
C_0(B)=C_{n,m,q_1,...,q_m}\[\prod_{i=1}^m2B^{\max(1,1/(q_i-1))+\max(1/2,1/(q_i-1))}+||\mu||_{\mathcal{SC}}^{m/2}\prod_{i=1}^mB^{1/(q_i-1)}\].
\end{align*}
which verifies the hypotheses of Theorem \ref{t:GM} for $S$.  Therefore for all $B>1$, there exists $C$ depending on $B,n,m,q_1,...,q_m$ such that
\begin{align*}
||S(f_1,...,f_m)||_{L^p(w^p)}&\leq C\prod_{i=1}^m||f_i||_{L^{p_i}(w_i^{p_i})}
\end{align*}
for all $1<p_1,...,p_m<\infty$, $w_i^{p_i}\in A_{w_i}$ with $[w_i^{p_i}]_{A_{p_i}}\leq B$, and $f_i\in L^{p_i}(w_i^{p_i})$.

\end{proof}

We now prove Theorem \ref{t:mconv}.

\begin{proof}
The implications {\it (iv)} $\Rightarrow$ {\it (iii)} $\Rightarrow$ {\it (ii)} $\Rightarrow$ {\it (i)} have already been proved in a more general context.  So it is sufficient to show that {\it (i)} $\Rightarrow$ {\it (iv)}.  Since $\theta_t(x,y_1,...,y_m)=t^{-mn}\Psi^t(t^{-1}(x-y_1),...,t^{-1}(x-y_m))$, it follows that $\Theta_t(1,...,1)(x)$ is constant constant in $x$:  For all $x\in\R^n$
\begin{align*}
\Theta_t(1,...,1)(x)&=\int_{\R^{mn}}t^{-mn}\Psi^t(t^{-1}(x-y_1),...,t^{-1}(x-y_m))d\vec y\\
&=\int_{\R^{mn}}\Psi^t(y_1,...,y_m)d\vec y=F(t)
\end{align*}
where the last line here we take as the definition of $F$.  But we have assumed that $\Theta_t$ satisfies the Carleson condition, and hence $|F(t)|^2\frac{dt}{t}dx$ is a Carleson measure.  So the strong Carleson condition follows:  For all cubes $Q\subset\R^n$
\begin{align*}
\int_0^{\ell(Q)}|\Theta_t(1,...,1)(x)|^2\frac{dt}{t}&=\frac{1}{|Q|}\int_Q\int_0^{\ell(Q)}|F(t)|^2\frac{dt}{t}dx\less1.
\end{align*}
If we assume also that $\Psi^t=\Psi$ is constant in $t$, then it follows that $F(t)=c_0$ is a constant function.  But then $|c_0|^2\frac{dt}{t}dx$ is a Carleson measure, and hence integrable on $Q\times(0,\ell(Q)]$ for all cubes $Q\subset\R^n$.  Then it follows that $c_0=0$ when $\Psi^t$ is constant in $t$, which completes the proof.
\end{proof}

\bibliographystyle{amsplain}

\end{document}